\documentclass[12pt,reqno,oneside]{amsart}

\usepackage{graphicx}
\usepackage{lscape}
\usepackage{amsfonts}
\usepackage{pdfsync}
\usepackage[toc,page]{appendix}
\usepackage{hyperref}
\usepackage{color}
\usepackage{amsmath}

\newtheorem*{theorem-non}{Theorem}
\newtheorem*{lemma-non}{Lemma}
\newtheorem{theorem}{Theorem}[section]
\newtheorem{lemma}[theorem]{Lemma}

\newtheorem{proposition}[theorem]{Proposition}
\newtheorem{example}[theorem]{Examples}

\theoremstyle{definition}
\newtheorem{defi}[theorem]{Definition}
\newtheorem{remark}[theorem]{Remark}
\newtheorem{theorem*}[theorem]{Theorem}

\def\O{\Omega}

\def\o{\omega}
\def\C{\mathcal{C}}

\def\d{{\rm d}}
\def\dist{{\rm dist}} 

\def\diam{{\rm diam}}
\def\length{{\rm length}}

\def\supp{{\rm supp}}
\def\dx{{\rm d}x}
\def\dy{{\rm d}y}
\def\dz{{\rm d}z}

\def\W{\mathcal{W}}

\def\S{\mathcal{S}}

\newcommand{\R}{{\mathbb R}}
\newcommand{\N}{{\mathbb N}}

\DeclareMathOperator*{\esssup}{ess\,sup}
\DeclareMathOperator*{\essinf}{ess\,inf}

\begin{document}
\date{\today}
\section*{}
\title[]{On the weighted fractional Poincar\'e-type inequalities}

\author[R. Hurri-Syrjanen]{Ritva Hurri-Syrj\"anen}
\address{Department of Mathematics and Statistics, FI-00014 University of Helsinki, Finland}
\email{ritva.hurri-syrjanen@helsinki.fi}

\author[F. L\'opez-Garc\'\i a]{Fernando L\'opez-Garc\'\i a}
\address{Department of Mathematics and Statistics\\ California State Polytechnic University Pomona, 
3801 West Temple Avenue, Pomona, CA (91768), US} 
\email{fal@cpp.edu}



\begin{abstract} 
Weighted fractional Poincar\'e-type  inequalities are proved on John domains
whenever the weights  defined on the domain  depend on the distance to 
the boundary and to  an arbitrary compact set in the boundary
of the domain.

\end{abstract}

\maketitle

\section{Introduction}
\label{Intro}
\setcounter{equation}{0}

In this article 
we study a version of the  classical fractional Poincar\'e-type inequality
where the domain in the double integral in the Gagliardo seminorm is replaced by a smaller one:
\begin{equation}\label{Ritva-Antti}
\left(\int_{\Omega}|u(x)-{ u_{\Omega}}|^p \dx\right)^{1/p}\leq C
\left(\int_\Omega\int_{B(x,\tau d(x))}\frac{|u(x)-u(y)|^p}{|x-y|^{n+sp}}\dy\dx\right)^{1/p}.
\end{equation}
The parameter $\tau$ in the double integral belongs to $(0,1)$
and $d(x)$ denotes the distance from $x$ to $\partial \Omega$. 
The inequality \eqref{Ritva-Antti} was introduced in
\cite{HV}.
It is well-known that the fractional classical Poincar\'e inequality is valid for any bounded domain, while this new version
\eqref{Ritva-Antti}
depends on the geometry of the domain. In \cite{HV} it was proved that the inequality 
\eqref{Ritva-Antti}
is valid on John domains
and, hence, in particular on Lipschitz domains. An example  of a domain where the inequality
\eqref{Ritva-Antti} is not valid
was  also given.
We refer the reader to \cite{HV1} and \cite{DIV}  where the fractional Sobolev-Poincar\'e versions of \eqref{Ritva-Antti} are considered.
For a weighted version of \eqref{Ritva-Antti} where weights are power functions to the boundary we refer to \cite{DD}. 

The main result of our paper is the following theorem where the distance to an arbitrary set of the boundary has been added as a weight.

\begin{theorem}\label{Main} 
Let $\O$ in $\R^n$ be a bounded John domain and $1< p<\infty$. Given a compact set $F$ in $\partial \Omega$,  and the parameters $\beta\geq 0$ and  $s,\tau\in (0,1)$, there exists a constant $C$ such that 
\begin{multline}\label{Fractional}
\left(\int_{\Omega}|u(x)-{ u_{\Omega ,\omega }}|^p d_F^{p\beta}(x)\dx\right)^{1/p}
\\
\leq C
\left(\int_\Omega\int_{B(x,\tau d(x))}\frac{|u(x)-u(y)|^p}{|x-y|^{n+sp}}  d^{ps}(x)d_F^{p\beta}(x)\dy\dx\right)^{1/p}
\end{multline}
for all functions $u\in L^p(\O,d(x)^{p\beta})$, where $d(x)$ and $d_F(x)$ denote the distance from $x$ to $\partial\O$ and $F$ respectively, and $u_{\Omega ,\omega}$ is the weighted average $\frac{1}{d_F^{p\beta}(\O)}\int_\O u(z) d_F^{p\beta}(z)\dz$.

In addition, the constant $C$ in (\ref{Fractional}) can be written as 
\begin{align*}
C= C_{n, p, \beta}\, \tau^{s-n} K^{n+\beta},
\end{align*} 
where $K$ is the geometric constant introduced in (\ref{Boman tree}).
\end{theorem}

We would like to emphasize two points in this result:
The first one is that no extra conditions are required for the compact set $F$ in $\partial \Omega$.
The second point is that the estimate shows how the constant depends on the given $\tau$ and a certain geometric condition of the domain.

Some of the essential auxiliary parts for the proofs for weighted inequalities are from 
\cite{L1} and \cite{L2} where  a useful decomposition technique was introduced  by the second author.
Our work  was stimulated by the papers of Augusto  C. Ponce,
\cite{P1},
\cite{P2}, \cite{P3}, where  more general fractional Poincar\'e inequalities  for functions defined  on Lipschitz domains were investigated. 

The paper is organized as follows: In Section \ref{Notation}, we introduce some definitions and preliminary results. In Section \ref{Inequalities}, we show how to use decompositions of functions to extend the validity of certain inequalities
on ``simple domains", such as cubes, to more complex ones. We are interested in extending the results from cubes to John domains. In Section \ref{Cubesl}, we apply the results obtained in the previous section to estimate the constant in the unweighted version of \eqref{Fractional} on cubes. Especially we are interested in how the constant depends on $\tau$. This result is auxiliary of our main theorem but it might be of independent interest. In Section \ref{John}, we show the validity of the weighted fractional Poincare inequality studied in this paper with the estimate of the constant and a generalization to the type of inequalities considered by Ponce.

\section{Notation and preliminary results}
\label{Notation}
\setcounter{equation}{0}

Throughout the paper $\Omega$ in $\R^n$ is a bounded domain with $n\geq 2$,  $1< p<\infty$, and $1<q<\infty$ with $\frac{1}{p}+\frac{1}{q}=1$, unless otherwise stated.
  Moreover, given $\eta:\Omega\to\R$ a weight (i.e.,  a positive measurable function) and $1\leq r\leq \infty$, we denote by $L^r(\Omega,\eta)$ the space of Lebesgue measurable functions $u:\Omega\to\R$
equipped with the norm
\[\|u\|_{L^r(\Omega,\eta)}:=\left(\int_\Omega|u(x)|^r\eta(x)\,\d x\right)^{1/r}\]
if $1\leq r<\infty$, and 
\[\|u\|_{L^\infty(\Omega,\eta)}:=\esssup_{x\in\Omega} |u(x)\eta(x)|.\]
Finally, given a set $A$ we denote by $\chi_A(x)$ its characteristic function.

\begin{defi}\label{Definition decomposition} Let $\mathcal{C}$ be the space of constant functions 
from $\R^n$ to $\R$
and  $\{U_t\}_{t\in\Gamma}$ a collection of open subsets of $\Omega$ that covers $\Omega$ except for a set of Lebesgue measure zero; $\Gamma$ is an index set.  It also satisfies the additional requirement that for each $t\in\Gamma$ the set $U_t$ intersects a finite number of $U_s$ with $s\in\Gamma$. This collection $\{U_t\}_{t\in\Gamma}$ is called an {\it open covering of $\Omega$}. Given $g\in L^1(\Omega)$ orthogonal to $\C$ (i.e., $\int g\, \varphi=0$ for all $\varphi\in\C$), we say that a collection of functions $\{g_t\}_{t\in\Gamma}$
  in $L^1(\Omega)$ is a $\C$-{\it orthogonal decomposition of $g$} subordinate to $\{U_t\}_{t\in\Gamma}$ if the following three properties are satisfied:
\begin{enumerate}
\item $g=\sum_{t\in \Gamma} g_t.$
\item $\supp (g_t)\subset U_t.$
\item $\int_{U_t} g_t\,\varphi=0$, for all $\varphi\in\C$ and $t\in\Gamma$.
\end{enumerate}
\end{defi}
We also refer to this collection of functions by a $\C$-{\it decomposition}.  We say that $\{g_t\}_{t\in\Gamma}$ is a {\it finite $\C$-decomposition} if $g_t\not\equiv 0$ only for a finite number of $t\in\Gamma$.

 Inequality (\ref{Fractional}), and similar Poincar\'e type inequalities, can be written in terms of a distance to the space of constant functions $\C$ by replacing its left hand side by 
\begin{equation*}
\inf_{\alpha\in \C}\left(\int_{\Omega}|u(x)-\alpha|^p d_F^{p\beta}(x)\dx\right)^{1/p}.
\end{equation*}
The technique used in this paper may also be considered when the distance to other vector spaces $\mathcal{V}$ are involved, in which case, a $\mathcal{V}$-orthogonal decomposition of functions is required. We direct the reader to \cite{L3} where a generalized version of the Korn inequality is studied by using decomposition of functions.

Let us denote by $G=(V,E)$ a graph with vertices $V$ and edges $E$. Graphs in this  paper   have neither multiple edges nor loops and the number of vertices in $V$ is at most countable. 

A {\it rooted tree} (or simply a tree) is a connected graph $G$ in which any two vertices are connected by exactly one simple path, and a {\it root} is simply a distinguished vertex $a\in V$. Moreover, if $G=(V,E)$ is a rooted tree with a root $a$, it is possible to define a {\it partial order} ``$\preceq$" in $V$ as follows: 
$s\preceq t$ if and only if the unique path connecting $t$ with the root $a$ passes through $s$. \label{level} The {\it height} or {\it level} of any $t\in V$ is the number of vertices in $\{s\in V\,:\,s\preceq t\text{ with }s\neq t\}$. {\it The parent} of a vertex $t\in V$ is the vertex $s$ satisfying that $s\preceq t$ and its height is one unit smaller than the height of $t$. We denote the parent of $t$ by $t_p$. 
It can be seen that each $t\in V$ different from the root has a unique parent, but several elements in $V$ could have the same parent. Note that two vertices are connected by an edge ({\it adjacent vertices}) if one is the parent of the other.

\begin{defi} \label{Decomp of Omega}
Let $\Omega$ be  in $\R^n$ be a bounded domain. We say that an open covering $\{U_t\}_{t\in\Gamma}$ is
a {\it tree covering} of $\Omega$ if it also satisfies the properties: 
\begin{enumerate}
\item $\chi_\Omega(x)\leq \sum_{t\in\Gamma}\chi_{U_t}(x)\leq N \chi_\Omega(x)$, for almost every $x\in\Omega$, where $N\geq 1$.
\item  $\Gamma$ is the set of vertices of a rooted tree $(\Gamma,E)$ with a
 root $a$.
\item There is a collection $\{B_t\}_{t\neq a}$ of pairwise disjoint open cubes with $B_t\subseteq U_t\cap U_{t_p}$.
\end{enumerate}
\end{defi}

\begin{defi}\label{T and W} 
Given a tree covering $\{U_t\}_{t\in\Gamma}$ of $\O$ we define the following {\it Hardy-type operator} $T$
on $L^1$-functions:
\begin{eqnarray}\label{Ttree}
Tg(x):=\sum_{a\neq t\in\Gamma}\dfrac{\chi_{t}(x)}{|W_t|}\int_{W_t}|g|,
\end{eqnarray} 
where
 \begin{equation}\label{W_t}
 \displaystyle{W_t:=\bigcup_{s\succeq t} U_s}\,,
 \end{equation}
and $\chi_t$ is the characteristic function of $B_t$ for all $t\neq a$. 
\end{defi}
We may refer to $W_t$ by {\it the shadow of }$U_t$. 

Note that the definition of $T$ is based on the a-priori choice of a tree covering $\{U_t\}_{t\in\Gamma}$ of $\O$. Thus, whenever $T$ is mentioned in this paper there is a tree covering $\{U_t\}_{t\in\Gamma}$ of $\O$ explicitly or implicitly associated to it.

The following fundamental result was proved in \cite[Theorem 4.4]{L2}, which shows the existence of a $\C-$decomposition of functions subordinate to a tree covering of the domain.

\begin{theorem}\label{Decomp} Let $\Omega$ in $\R^n$ be a bounded domain with a tree covering $\{U_t\}_{t\in \Gamma}$. Given $g\in L^1(\Omega)$ such that  $\int_\O g\varphi=0$, for all $\varphi\in\C$, and $\supp(g)\cap U_s\neq \emptyset$ for a finite number of $s\in\Gamma$, there exists a  $\C$-decompositions $\{g_t\}_{t\in\Gamma}$ of $g$ subordinate to $\{U_t\}_{t\in\Gamma}$  (refer to Definition \ref{Definition decomposition}).

Moreover, let $t\in\Gamma$. If $x\in B_s$ where $s=t$ or $s_p=t$ then 
\begin{eqnarray}\label{P02}
|g_t(x)|\leq |g(x)|+\tfrac{|W_s|}{|B_s|}Tg(x),
\end{eqnarray}
where $W_t$ denotes the shadow of $U_t$  defined in  \eqref{W_t}.  Otherwise
\begin{eqnarray}\label{P01}
|g_t(x)|\leq |g(x)|.
\end{eqnarray}
\end{theorem}

 \begin{remark}\label{finite decomp}
The $\C$-decomposition stated in Theorem \ref{Decomp} is finite. This fact is not in the statement of \cite[Theorem 4.4]{L2} but it is easily deduced from its proof. 
\end{remark}

In the next lemma, the continuity of the operator $T$ is shown. We refer the reader to \cite[Lemma 3.1]{L1} for its proof.

\begin{lemma}\label{Ttreecont} The operator $T:L^q(\O)\to L^q(\O)$ defined in (\ref{Ttree})
  is continuous for any $1<q\leq \infty$. Moreover, its norm is bounded by 
\[\|T\|_{L^q\to L^q} \leq 2\left(\dfrac{qN}{q-1}\right)^{1/q}.\]
Here $N$ is the overlapping constant from Definition \ref{Decomp of Omega}.
\end{lemma}

If $q=\infty$, the previous inequality means $\|T\|_{L^\infty\to L^\infty} \leq 2$. Actually, for being $T$ an averaging operator, it can be easily observed that $\|T\|_{L^\infty\to L^\infty} = 1$, but it does not affect our work. Notice that $L^q(\Omega, \omega^{-q})\subset L^1(\Omega)$ if the weight $\omega:\Omega\to\R_{>0}$ satisfies that $\omega^p\in L^1(\Omega)$. Then, the operator $T$ introduced in Definition \ref{T and W} for functions in $L^1(\Omega)$ is well-defined in $L^q(\Omega, \omega^{-q})$.

\begin{lemma}\label{weighted T}
Let $\O$ in $\R^n$ be a bounded domain, $\{U_t\}_{t\in\Gamma}$ a tree covering of $\Omega$ and $\omega:\Omega\to\R$ a weight which satisfies $\omega^p\in L^1(\Omega)$. If $\omega$ satisfies that 
\begin{align}\label{weight}
\esssup_{y\in W_t} \omega(y) \leq C_2 \essinf_{x\in B_t} \omega(x),
\end{align}
 for all $a\neq t\in\Gamma$, then the Hardy-type operator $T$ defined in (\ref{Ttree}) and subordinate to $\{U_t\}_{t\in\Gamma}$ is continuous from $L^q(\O,\omega^{-q})$ to itself. Moreover, its norm for $1<q<\infty$ is bounded by 
\[\|T\|_{L\to L} \leq 2\left(\frac{qN}{q-1}\right)^{1/q}C_2,\]
where $L$ denotes $L^q(\O,\omega^{-q})$, and $N$ is the overlapping constant from Definition 
\ref{Decomp of Omega}.
\end{lemma}

\begin{proof} Given $g\in L^q(\O,\omega^{-q})$ we have
\begin{eqnarray*} 
& &\int_\O|Tg(x)|^q \omega^{-q}(x)\,\d x\\
&=&\int_{\O} \omega^{-q}(x) \left|\sum_{a\neq t\in\Gamma}\frac{\chi_t(x)}{|W_t|}\int_{W_t}|g(y)|\,\d y\right|^q\,\d x\\
&=& \int_{\O} \left|\sum_{a\neq t\in\Gamma}\frac{\chi_t(x)}{|W_t|} \omega^{-1}(x)\int_{W_t}|g(y)|\omega^{-1}(y)\,\omega(y)\,\d y\right|^q\,\d x.
\end{eqnarray*}

Now, condition (\ref{weight}) implies that $\omega(y)\leq C_2 \omega(x)$ for almost every $x\in B_t$ and $y\in W_t$.  Thus, 
\begin{eqnarray*} 
& &
\int_\O|Tg(x)|^q \omega^{-q}(x)\,\d x\\
&\leq& \int_{\O} \left|\sum_{a\neq t\in\Gamma}\frac{\chi_t(x)}{|W_t|} \omega^{-1}(x)\,C_2\,\omega(x)\int_{W_t}|g(y)|\omega^{-1}(y)\,\d y\right|^q\,\d x\\
&=& C_2^q \int_{\O} \left|\sum_{a\neq t\in\Gamma}\frac{\chi_t(x)}{|W_t|}\int_{W_t}|g(y)|\omega^{-1}(y)\,\d y\right|^q\,\d x\\
&=&C_2^{q} \int_{\O} \left|T(g\omega^{-1})\right|^q\,\d x.
\end{eqnarray*}

Finally, $g\omega^{-1}$ belongs to $L^q(\O)$ and $T$ is continuous from $L^q(\O)$ to itself; we refer to  Lemma \ref{Ttreecont}, 
hence
\begin{eqnarray*} 
\int_\O|Tg(x)|^q \omega^{-q}(x)\,\d x
\leq \left(2^q \frac{qN}{q-1}\right)C_2^q\, \|g\|^q_{L^q(\O,\omega^{-q})}.
\end{eqnarray*}
\end{proof}

\section{A decomposition and Fractional Poincar\'e inequalities}
\label{Inequalities}

Let $\Omega$ in $\R^n$ be an arbitrary bounded domain and $\{U_t\}_{t\in\Gamma}$ an open covering of $\O$. The weight $\omega:\Omega\to\R_{>0}$ satisfies that $\omega^p\in L^1(\Omega)$. In addition, $u_\Omega$ denotes the average $\frac{1}{|\Omega|}\int_\Omega u(z)\dz$. For weighted spaces of functions, $u_{\Omega ,\omega}$ represents the weighted average $\frac{1}{\omega(\O)}\int_\Omega u(z)\omega(z)\dz$, where $\omega(\O):=\int_\O\omega(z)\dz$.

Now, given a bounded domain $U$ in $\R^n$ and a nonnegative measurable function
$\mu:U\times U\to \R$ we introduce the following Poincar\'e type inequality 
\begin{align}\label{local Frac}
\inf_{c\in\R}\|u-c\|_{L^p(U,\omega^p)}\leq  C \left(\int_U\int_U|u(x)-u(y)|^p\mu(x,y)\,\d y\d x\right)^{1/p},
\end{align}
where $u\in L^p(U,\omega^p)$. Notice that the right hand side in this inequality might be  infinite. The validity of (\ref{local Frac}) depends on $U$, $p$, $\mu$ and $\omega$. The function $\mu(x,y)$ might be zero, however,  $\omega(x)$ is strictly positive almost everywhere in $\Omega$.

Let us mention three examples.

\begin{example}\label{three_examples}
\begin{enumerate}
\item The weighted fractional Poincar\'e inequality  
with $\mu(x,y)=\frac{1}{|x-y|^{n+sp}}$, where 
$s\in (0,1)\,,$
is the classical fractional Poincar\'e inequality which is clearly valid for any arbitrary bounded domain. \\
\item If $\mu(x,y)=\frac{\chi_{B_x}(y)}{|x-y|^{n+sp}}$, where $B_x$ is the ball centered at $x$ with radius $\tau d(x)$ for $s,\tau\in (0,1)$, then the inequality represents a more recently studied fractional Poincar\'e inequality whose validity depends on the geometry of the domain (refer to \cite{HV} for details). \\
\item Finally, $\mu(x,y)=\frac{\rho(|x-y|)}{|x-y|^p}$, where $\rho$
 is a certain nonnegative radial function, is another inequality which has also been studied recently  (refer to \cite{P1} for details).
\end{enumerate}
\end{example}

 Inequality \eqref{local Frac} deals with an estimation of the distance to $\C$ of an arbitrary function $u$ in $L^p(\Omega,\omega^p)$. The local-to-global argument used in this paper to study this Poincar\'e type inequalities is based on the fact that $L^p(\Omega,\omega^p)$ is the dual space of $L^q(\Omega,\omega^{-q})$ and the existence of decompositions of functions in $L^q(\Omega,\omega^{-q})$ orthogonal to $\C$. Let us properly define this set and a subspace:
\begin{align}\label{space W}
\W&:=\{g\in L^q(\Omega,\o^{-q})\,:\, \int g\varphi=0 \text{ for all } \varphi\in\C \}\\ 
\label{space W0}
\W_0&:=\{g\in \W\,:\, \supp(g)\text{ intersects a finite number of } U_t \}.
\end{align}

The integrability of $\omega^p$ implies that $L^q(\Omega,\o^{-q})\subset L^1(\Omega)$, then $\W$ and $\W_0$ are well-defined.  Following Remark \ref{finite decomp}, the $\C$-decomposition of functions in $\W_0$ stated in Theorem \ref{Decomp} is finite, which is not valid in general for functions in $\W$. This property verified by the functions in $\W_0$ simplifies the proof of Lemma \ref{Decomp Frac}, which motivates the definition of this space.

Now, we introduce the spaces
\begin{align}
\W\oplus \omega^p\C&=\{g+\alpha\omega^p\,/\,g\in \W\text{ and }\alpha\in\C\}\notag\\
\label{space S}
\S:=\W_0\oplus \omega^p\C&=\{g+\alpha\omega^p\,/\,g\in \W_0\text{ and }\alpha\in\C\}.
\end{align}
It is not difficult to observe that $L^q(\Omega,\omega^{-q})=\W\oplus \omega^p\C$ and $\S$ is a subspace of $L^q(\Omega,\omega^{-q})$. The following lemma, which was proved in \cite[Lemma 3.1]{L2}, states that $\S$ is also dense  in $L^q(\Omega,\omega^{-q})$ and uses in its proof the requirement that says that for each $t\in\Gamma$ the set $U_t$ intersects a finite number of $U_s$ with $s\in\Gamma$. 

\begin{lemma}\label{weighted dense} 
The space $\S$ 
is dense in $L^q(\Omega,\omega^{-q})$. Moreover,  if
$g+\alpha\o^p$ is an element in $\S$ then \[\|g\|_{L^q(\Omega,\o^{-q})}\leq 2\|g+\alpha\o^p\|_{L^q(\Omega,\o^{-q})}.\]
\end{lemma}

\begin{lemma}\label{Decomp Frac} If there exists an open covering $\{U_t\}_{t\in\Gamma}$ of $\O$ such that (\ref{local Frac}) is valid on $U_t$ for all $t\in\Gamma$, with  a  uniform constant $C_1$, and there exists a  finite $\C$-orthogonal decomposition of any function $g$ in $\W_0$ subordinate to $\{U_t\}_{t\in\Gamma}$, with the estimate
\begin{equation*}
\sum_{t\in \Gamma} \|g_t\|^q_{L^q(U_t,\o^{-q})}\leq C^q_0\|g\|^q_{L^q(\Omega,\o^{-q})},
\end{equation*}
then, there exists a constant $C$ such that  
\begin{equation}\label{lemma Frac}
\|u-u_{\Omega ,\omega }\|_{L^p(\Omega,\omega^p)}\leq  C \left(\sum_{t\in\Gamma}\int_{U_t}\int_{U_t}|u(x)-u(y)|^p\mu(x,y)\,\d y\d x\right)^{1/p}
\end{equation}
is valid for any $u\in L^p(\Omega,\omega^p)$. Moreover, the constant $C=2C_0C_1$ holds in (\ref{lemma Frac}).
\end{lemma}

\begin{proof}
Without loss of generality we can assume that 
$u_{\Omega ,\omega }=0$. 
We estimate the norm on the left hand side of the inequality by duality. Thus, let $g+\omega^p\psi$ be an arbitrary function in $\S$, we refer to Lemma \ref{weighted dense}. Then,  by using the  finite $\C$-orthogonal decomposition of $g$ we conclude that
\begin{eqnarray}\label{first}
\int_\Omega u(g+\alpha\omega^{p})&=&\int_\Omega ug=\int_\Omega u\sum_{t\in\Gamma} g_t\notag\\
&=&\sum_{t\in\Gamma} \int_{U_t} u g_t=\sum_{t\in\Gamma} \int_{U_t} (u-c_t)g_t.
\end{eqnarray}
Notice that the identity in the second line is valid for any $t\in\Gamma$ and $c_t\in \R$.
 
   Next, by using the H\"older inequality in $\eqref{first}$,  the fact that (\ref{local Frac}) is valid on $U_t$ with a uniform constant $C_1$ and, finally, the H\"older inequality over the sum, we obtain 
\begin{align*}
&\int_\Omega u(g+\alpha\omega^{p})
\leq \sum_{t\in\Gamma} \inf_{c\in\R}\|u-c\|_{L^p(U_t,\omega^{p})}\|g_t\|_{L^q(U_t,\omega^{-q})}\\
&\leq C_1\sum_{t\in\Gamma} \left(\int_{U_t}\int_{U_t}|u(x)-u(y)|^p\mu(x,y)\,\d y\d x\right)^{1/p}\|g_t\|_{L^q(U_t,\omega^{-q})}\\
&\leq C_1\left(\sum_{t\in\Gamma} \int_{U_t}\int_{U_t}|u(x)-u(y)|^p\mu(x,y)\,\d y\d x\right)^{1/p}\left(\sum_{t\in\Gamma} \|g_t\|^q_{L^q(U_t,\omega^{-q})}\right)^{1/q}\\
&\leq C_0C_1\left(\sum_{t\in\Gamma} \int_{U_t}\int_{U_t}|u(x)-u(y)|^p\mu(x,y)\,\d y\d x\right)^{1/p} \|g\|_{L^q(U,\omega^{-q})}\\
&\leq 2C_0C_1\left(\sum_{t\in\Gamma} \int_{U_t}\int_{U_t}|u(x)-u(y)|^p\mu(x,y)\,\d y\d x\right)^{1/p} \|g+\alpha \omega^p\|_{L^q(U,\omega^{-q})}.
\end{align*}
Finally, as $\S$ is dense in $L^q(\Omega,\omega^{-q})$, by taking the supremum over all the functions $g+\alpha\omega^p$ in $\S$ with $\|g+\alpha\omega^p\|_{L^q(\O,\omega^{-q})}\leq 1$ we prove the result. 
\end{proof}

\section{On fractional Poincar\'e inequalities on cubes}
\label{Cubesl}
\setcounter{equation}{0}

In this section, we use the results stated in the previous two sections to show a certain fractional Poincar\'e inequality on an arbitrary cube $Q$. Thus, in order to show the existence of the $\C$-decomposition, which is used later to apply Lemma \ref{Decomp Frac}, we define a tree covering $\{U_t\}_{t\in\Gamma}$ of $Q$. This covering is only used in this section and for cubes. In the following section, we work with a different bounded domain, an arbitrary bounded John domain, which requires a different covering. However, let us warn the reader that we will keep the notation $\{U_t\}_{t\in\Gamma}$ used in Section \ref{Inequalities}.

The validity of the local inequality stated in the following proposition is well-known. We refer the reader to \cite{DD} for its proof.

\begin{proposition}\label{DD} The fractional Poincar\'e inequality 
\begin{align*}
\inf_{c\in\R} \|u(x)-c\|_{L^p(U)}\leq \left(\frac{{\rm diam}\,(U)^{n+sp}}{|U|}\int_U\int_U \frac{|u(y)-u(x)|^p}{|y-x|^{n+sp}}\,\d y\d x\right)^{1/p}
\end{align*}
holds for any bounded domain $U$ in $\R^n$ and $1\leq p<\infty$.
\end{proposition}

The following proposition is a special case of \cite[Lemma 2.2]{HV}. In the present paper, we give a different proof which let us estimate the dependance of the constant with respect to $\tau$.

\begin{proposition}\label{on cubes} Let $Q$ in $\R^n$ be a cube with side length $l(Q)=L$, $1 < p<\infty$ and $\tau\in (0,1)$. Then, the following inequality holds 
\begin{align*}
\inf_{c\in\R} \|u(x)-c\|_{L^p(Q)}\leq C_{n,p}\, \tau^{s-n} L^s\left(\int_Q\int_{Q\cap B(x,\tau L)} \frac{|u(y)-u(x)|^p}{|y-x|^{n+sp}}\,\d y\d x\right)^{1/p},
\end{align*}
where $C_{n,p}$ depends only on $n$ and $p$.
\end{proposition}

\begin{proof} This result follows from Lemma \ref{Decomp Frac} on the cube $Q$, where $\mu(x,y)=\frac{1}{|x-y|^{n+sp}}$ and $\omega\equiv 1$. So, let us start by defining an appropriate tree covering of $Q$ to obtain, via Theorem \ref{Decomp} and Remark \ref{finite decomp}, a  finite $\C$-decomposition of any functions in $\W_0$. Let $m\in \N$ be such that $\frac{\sqrt{n+3}}{\tau}<m\leq 1+\frac{\sqrt{n+3}}{\tau}$ and $\{A_t\}_{t\in\Gamma}$ the regular partition of $Q$ with $m^n$ open cubes. The side length of each cube is $l(A_t)=\frac{L}{m}$. In the example shown in Figure \ref{Fig tree},  $m=4$ and the index set  $\Gamma$ has 16 elements. 
\begin{figure}[htb]
       \includegraphics[width=50mm]{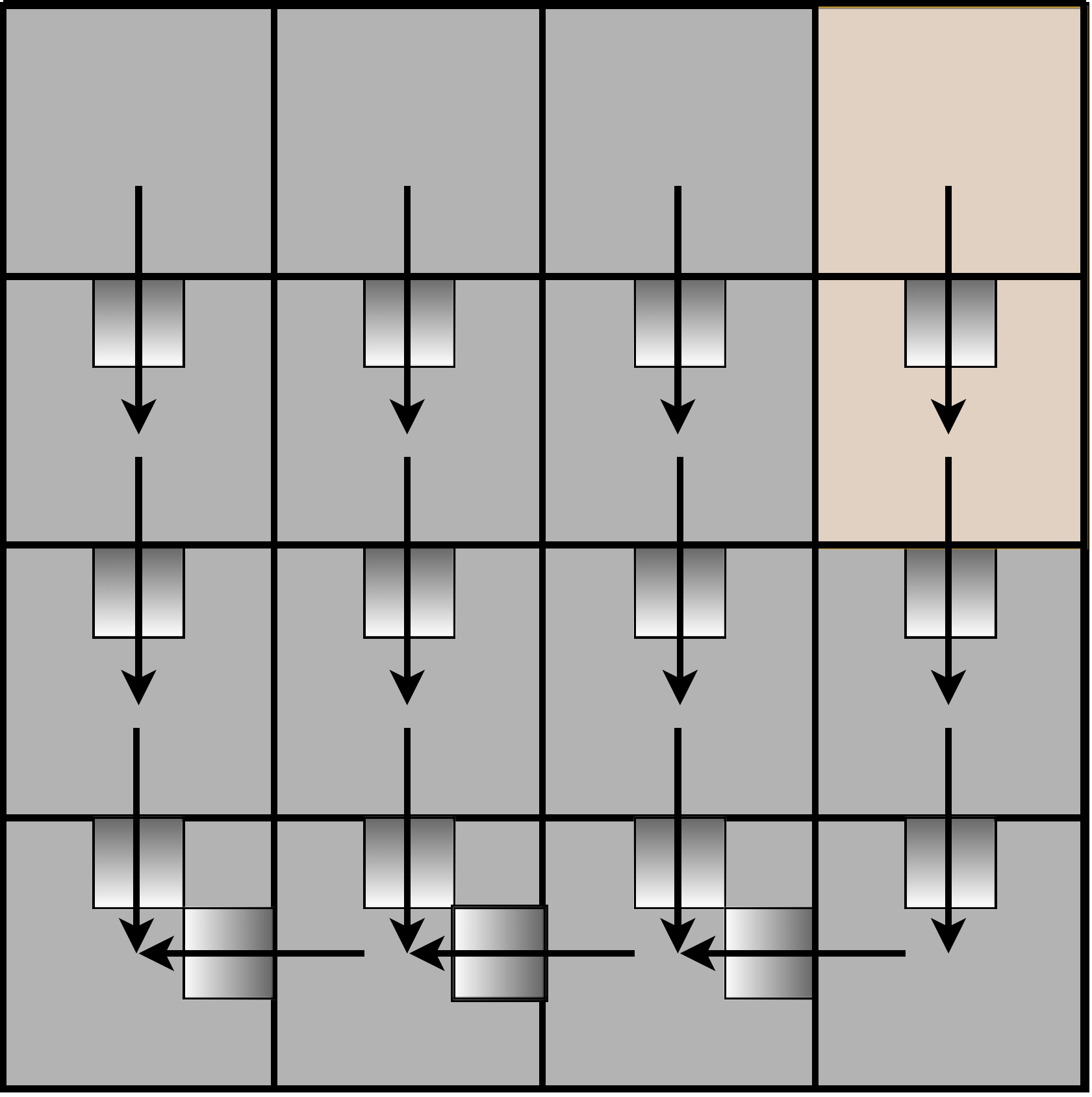}
       \caption{A tree covering of $Q$}
       \label{Fig tree}
\end{figure}

The tree covering of $Q$ that we are looking for will be defined by enlarging  the sets in the covering $\{A_t\}_{t\in\Gamma}$ in an appropriate way but keeping the tree structure of $\Gamma$, which is introduced in the following lines. Indeed, we pick a cube $A_a$, whose index will be the root, and inductively define a tree structure in $\Gamma$ such that the unique chain connecting $t$ with $a$ is associated to a chain of cubes connecting $Q_t$ with $Q_a$, with minimal number of cubes, such that two consecutive cubes share a $n-1$ dimensional face. In Figure \ref{Fig tree}, the cube $A_a$ is in the lower left corner and the tree structure is represented  using black arrows that``descend" to the root. Now that $\Gamma$ has a tree structure, we define the tree covering $\{U_t\}_{t\in\Gamma}$ of $Q$  with the rectangles $U_t:=(\overline{A_t}\cup \overline{A_{t_p}})^\circ$ if $t\neq a$ and $U_a:=A_a$. In order to have a better understanding of the construction, notice that $U_t\cap U_{t_p}=A_{t_p}$ for all $t\neq a$. Moreover, the index set $\Gamma$ in the example with its tree structure has 7 levels, from level 0 to level 6 (refer to page \pageref{level} for definitions), with only one index of level 6, whose rectangle $U_t$ appears in Figure \ref{Fig tree} in a different color. 

Now, let us define the collection $\{B_t\}_{t\neq a}$ of pairwise disjoint open cubes $B_t\subseteq U_t\cap U_{t_p}$  or equivalently $B_t\subseteq A_{t_p}$. Given $t\neq a$, we split $A_{t_p}$ into $3^n$ cubes with the same size. The open set $B_t$ is the cube in the regular partition of $A_{t_p}$ whose closure intersects the $n-1$ dimensional face $A_{t_p}$ in the intersection $(\overline{A_t}\cap \overline{A_{t_p}})$. There are  $3^{n-1}$ cubes with that property but we pick  $B_t$ to be the  one  which does not share any part of any other $n-1$ dimensional face of $\overline{A_{t_p}}$.

The cubes in $\{B_t\}_{t\neq a}$ have side length equal to $\frac{L}{3m}$ and are represented in Figure \ref{Fig tree} by the 15 grey gradient small cubes.
By its construction, it is easy to check that $\{B_t\}_{t\neq a}$ is a collection of pairwise disjoint open cubes $B_t\subseteq U_t\cap U_{t_p}$, hence, $\{U_t\}_{t\in\Gamma}$ is a tree covering of $Q$ with $N=2n$ (it could also be  less).  

By Theorem \ref{Decomp}, there is a finite $\C$-decomposition of functions $\{g_t\}_{t\in\Gamma}$ subordinate to $\{U_t\}_{t\in\Gamma}$ which satisfies (\ref{P02}) and $(\ref{P01})$. Moreover, it can be seen that  
\[\frac{|W_s|}{|B_s|}\leq \frac{|Q|}{|B_s|}= (3m)^n,\]
for all $s\in\Gamma$, thus,
\begin{equation*}
|g_t(x)|\leq |g(x)|+(3m)^nTg(x),
\end{equation*}
for all $t\in\Gamma$ and $x\in U_t$. Next, using the continuity of $T$ stated in Lemma \ref{Ttreecont} and some straightforward calculations we conclude 
\begin{align*}
\sum_{t\in \Gamma} \|g_t\|^q_{L^q(U_t)}&\leq 2^{q-1}N\left(1+(3m)^{nq}2^q\frac{qN}{q-1}\right)\|g\|^q_{L^q(Q)}\\
&\leq \frac{2^{2q+2}n^2q}{q-1}\,(3m)^{nq}\,\|g\|^q_{L^q(Q)}\\
&\leq \frac{2^{2q+2}3^{nq}n^2q}{q-1}\,\left(1+\sqrt{n+3}\right)^{nq}\,\tau^{-nq}\,\|g\|^q_{L^q(Q)}.
\end{align*}
Hence, we have  a finite $\C$-decomposition of any function in $\W_0$ subordinate to $\{U_t\}_{t\in\Gamma}$ with the constant in the estimate equal to  
\begin{align*}
C_0&=\left(\frac{2^{2q+2}{3^{nq}}n^2q}{q-1}\right)^{1/q}\,\left(1+\sqrt{n+3}\right)^n\,\tau^{-n}.
\end{align*}
Now, from Proposition \ref{DD} and using  that $m>\frac{\sqrt{n+3}}{\tau}$, we can conclude that inequality (\ref{local Frac}) is valid on each $U_t$ with an uniform constant 
\begin{align*}
C_1&=(n+3)^{n/2p}(\tau L)^{s}.
\end{align*}
Thus, using Lemma \ref{Decomp Frac} we can claim that 
\begin{align*}
\|u-u_{Q}\|_{L^p(Q)}\leq 2C_0C_1 \left(\sum_{t\in\Gamma}\int_{U_t}\int_{U_t}\frac{|u(x)-u(y)|^p}{|x-y|^{n+sp}}\,\d y\d x\right)^{1/p}.
\end{align*}
Finally, $\text{diam}(U_t)\leq \sqrt{n+3}\frac{L}{m}\leq \tau L$, thus $U_t\subset B(x,\tau L)$ for  any $x\in U_t$, thus, using the control on the overlapping of the tree covering given by $N=2n$, it follows that
\begin{align*}
\|u-u_{Q}\|_{L^p(Q)}\leq C_{n,p}\, \tau^{-n} (\tau L)^s\, \left(\int_{Q}\int_{Q\cap B(x,\tau L)}\frac{|u(x)-u(y)|^p}{|x-y|^{n+sp}}\,\d y\d x\right)^{1/p},
\end{align*}
where 
\begin{align}\label{C on cubes}
C_{n,p}=2\left(\frac{2^{2q+2}3^{nq}n^2q}{q-1}\right)^{1/q}\,\left(1+\sqrt{n+3}\right)^n(n+3)^{n/2p}(2n)^{1/p}.
\end{align}
\end{proof}

\section{On fractional Poincar\'e inequalities on John domains}
\label{John}
\setcounter{equation}{0}

In this section, we apply the results obtained in the previous sections on an arbitrary bounded John domain $\Omega$. Its definition is recalled below. The weight $\omega(x)$ is defined as $d_F(x)^\beta$, where $d_F(x)$ denotes the distance from $x$ to an arbitrary compact set $F$ in $\partial\O$ and $\beta\geq 0$. In the particular case where $F=\partial\Omega$, $d_{\partial\Omega}(x)$ is simply denoted as $d(x)$. Notice that $\omega^p$ belongs to $L^1(\Omega)$ for being $\Omega$ bounded and $\beta$ nonnegative.

A Whitney decomposition of $\O$ is a collection $\{Q_t\}_{t\in\Gamma}$ of closed pairwise disjoint dyadic cubes, which verifies
\begin{enumerate}
\item $\O=\bigcup_{t\in\Gamma}Q_t$.
\item $\text{diam}(Q_t) \leq \text{dist}(Q_t,\partial\Omega) \leq 4\text{diam}(Q_t)$.
\item $\frac{1}{4}\text{diam}(Q_s)\leq \text{diam}(Q_t)\leq 4\text{diam}(Q_s)$, if $Q_s\cap Q_t\neq \emptyset$.
\end{enumerate}
Here,  $\dist (Q_t,\partial \Omega )$ is the Euclidean distance between $Q_t$ and the boundary of $\Omega$, denoted by $\partial \Omega$.
The diameter of the cube $Q_t$ is denoted by $\diam (Q_t)$ and  the side length is written as $\ell (Q_t)$.

Two different cubes $Q_s$ and $Q_t$ with $Q_s\cap Q_t\neq \emptyset$ are called {\it neighbors}. This kind of covering exists for any proper open set in $\R^n$ (refer to  \cite[VI 1]{S} for details). Moreover, each cube $Q_t$ has less than or equal to $12^n$ neighbors. And, if we fix $0<\epsilon<\frac{1}{4}$ and define $(1+\epsilon )Q_t$ as the cube with the same center as $Q_t$ and side length $(1+\epsilon)$ times the side length of $Q_t$, then $(1+\epsilon )Q_t$ touches $(1+\epsilon )Q_s$ if and only if  $Q_t$ and $Q_s$ are neighbors. 

Given a Whitney decomposition $\{Q_t\}_{t\in\Gamma}$ of $\O$ we refer by an expanded Whitney decomposition of $\O$ to the collection of open cubes $\{Q^*_t\}_{t\in\Gamma}$ defined by 
\begin{align*}
Q^*_t:=\frac98 Q_t^\circ.
\end{align*}
Observe that this collection of cubes satisfies that \[\chi_\O(x) \leq 12^n \sum_{t\in\Gamma}\chi_{Q^*_t}(x)\leq (12^n)^2  \chi_\O(x)\]
for all $x\in\R^n.$

We recall the definition of a {\it bounded John domain}. 
A bounded domain $\Omega$ in $\R^n$ is a John domain with constants $a$ and $b$, $0<a\le b <\infty$, if
there is a point $x_0$ in $\Omega$ such that for each point $x$ in $\Omega$ there exists a rectifiable curve
$\gamma_x$ in $\Omega$, parametrized by its arc length written as  $\length (\gamma_x )$, such that
\[
\dist (\gamma_x (t),\partial \Omega )\geq \frac{a}{\length (\gamma_x)}t\quad \mbox{ for all } t\in [0,\length (\gamma_x)]
\]
and
\[
\length (\gamma_x)\le b.
\]
Examples of John domains are convex domains, uniform domains, and also domains with slits, for example
$B^2(0,1)\backslash [0,1)$.
The John property fails in domains with zero angle outward spikes. 
John domains were introduced by Fritz John in \cite{J}  and they were renamed by O. Martio and J. Sarvas as John domains later.

There are other equivalent definitions of John domains. In these notes, we are interested in a definition of the style of Boman chain condition (see \cite{BKL}) in terms of Whitney decompositions and trees. This equivalent definition 
is  introduced in \cite{L2}.

\begin{defi}\label{tcovJohn} A bounded domain $\O$ in $\R^n$ is a John domain if for any Whitney decomposition $\{Q_t\}_{t\in\Gamma}$, there exists a constant $K>1$ and a tree structure of $\Gamma$, with  a root $a$, 
that satisfies 
\begin{align}\label{Boman tree}
Q_s\subseteq KQ_t,
\end{align}
for any $s,t\in\Gamma$ with $s\succeq t$. In other words, the shadow of 
$Q_t$  written as  $W_t$ is contained in $KQ_t$; refer to \eqref{W_t}. Moreover, the intersection of the cubes associated to adjacent indices, $Q_t$ and $Q_{t_p}$, is an $n-1$ dimensional face of one of these  cubes.
\end{defi}

Now, given a Whitney decomposition $\{Q_t\}_{t\in\Gamma}$ of a bounded John domain $\Omega$ in $\R^n$, with constant $K$ in the sense of (\ref{Boman tree}), we define the tree covering 
 $\{U_t\}_{t\in\Gamma}$
of expanded Whitney cubes
such that
\begin{align}\label{cover John}
U_t:= Q_t^*.
\end{align}
The overlapping is bounded by $N=12^n$. 
Now, each open cube $B_t$ in the collection $\{B_t\}_{t\neq a}$ shares the center with the $n-1$ dimensional face $Q_t\cap Q_{t_p}$ and has side length $\frac{l_t}{64}$, where $l_t$ is the side length of $Q_t$. It follows from the third condition in the Whitney decomposition, and some calculations, that this collection is pairwise disjoint and \[B_t\subset Q_t^*\cap Q_{t_p}^*=U_t\cap U_{t_p}.\]
Moreover, it can be seen that 
\begin{align}\label{Eccentricity}
\frac{|W_t|}{|B_t|}\leq \frac{(K\frac{9}{8}l_t)^n}{(\frac{l_t}{64})^n}=72^n\,K^n,
\end{align}
for all $t\in\Gamma$, with $t\neq a$.

\begin{lemma}\label{dist F} Let $\Omega$ in $\R^n$ be a John domain with the constant $K$ in the sense of (\ref{Boman tree}), $F$ in $\partial\Omega$ a compact set and $d_F(x)$ the distance from $x$ to $F$. Then, 
\begin{align*}
\sup_{y\in W_t} d_F(y) \leq 3K\sqrt{n} \inf_{x\in B_t} d_F(x),
\end{align*}
for all $t\in \Gamma$.
\end{lemma}

A similar inequality is also valid if we consider the weight $d_F^{\beta}(x)$ with a nonnegative power of the distance to $F$. Thus, this lemma implies, via Lemma \ref{weighted T}, the continuity of the operator $T$ from $L^q(\Omega,d_F^{-q\beta})$ to itself  with an estimation of its constant. 
Then, there exists a $\C$-decomposition with a weighted estimate for a certain weight. 

\begin{proof} Given $t\in\Gamma$, with $t\neq a$, $x\in B_t$ and $y\in W_t:=\cup_{s\succeq t}U_s$, we have to prove that $d_F(y)\leq 3K d_F(x)$. Notice that $d(x)\leq d_F(x)$ for all $x\in\Omega$.  Moreover, $Q_s\subseteq K Q_t$ for all $s\succeq t$, then $W_t\subseteq K U_t$.  In addition, 
\begin{align*}
d_F(y)&\leq |y-x| + d_F(x)\leq \text{diam}(W_t) + d_F(x)\\
& \leq K\text{diam}(U_t) + d_F(x)\\ 
& = K\frac{_9}{^8}\text{diam}(Q_t)+d_F(x).\\
\end{align*}
Finally, using the second property stated in the Whitney decomposition it follows that $3Q_t\subset \Omega$. Then, as  
\[{\rm dist}(Q^*_t,\partial\O)\geq {\rm dist}(Q^*_t,(3Q_t)^c)\geq \frac{_{15}}{^{16}}l_t,\] 
doing some calculations we can assert that 
\[\text{diam}(Q_t)\leq \frac{_{16}}{^{15}}\sqrt{n}\,\text{dist}(Q_t^*,\partial\Omega)\leq \frac{_{16}}{^9}\sqrt{n}\,\text{dist}(Q_t^*,\partial\Omega).\] 
Thus, 
\begin{align*}
d_F(y)& \leq 2K\sqrt{n}\,\text{dist}(Q^*_t,\partial \Omega)+d_F(x)\\
&\leq 2K\sqrt{n}\,d(x)+d_F(x)\leq 2K\sqrt{n}\, d_F(x) + d_F(x).
\end{align*}
\end{proof}

Now we are able to prove Theorem \ref{Main} and also to give  the dependence of the constant $C$ on the given value of $\tau$ and the constant $K$ from \eqref{Boman tree}.

\begin{proof}[Proof of Theorem \ref{Main}] This result follows from Lemma \ref{Decomp Frac} with the tree covering $\{U_t\}_{t\in\Gamma}$ of $\Omega$ defined in (\ref{cover John}), $\omega(x):=d^\beta_F(x)$ and 
\begin{align}\label{mu}
\mu(x,y):=\frac{d^{ps}(x)\, d_F^{p\beta}(x)\, \chi_{B(x,\tau d(x))}(y)}{|x-y|^{n+sp}}.
\end{align}
Notice that $\omega^{p}$ belongs to $L^1(\Omega)$, the condition assumed at the beginning of Section \ref{Inequalities}. The validity of (\ref{local Frac}) on  a cube $U_t$, with  a uniform constant $C_1$, follows from Proposition \ref{on cubes}. Indeed, by using 
 the fact
that $U_t$ is 
an expanded Whitney cube by a factor $9/8$  and $F\subseteq\partial \Omega$, it follows that  \[\sup_{x\in U_t}d^\beta_F(x)\leq 2^\beta \inf_{x\in U_t}d^\beta_F(x).\]
Thus, we have 
 \begin{align*}
\inf_{c\in\R} \|u(x)-c\|&_{L^p(U_t,d_F^{p\beta})}\\
\leq &C_{n,p}\, \tau^{s-n} L_t^s\, 2^\beta\left(\int_{U_t}\int_{U_t} \frac{|u(x)-u(y)|^p}{|x-y|^{n+sp}} d_F^{p\beta}(x)\chi_{B(x,\tau L_t)}(y)\,\d y\d x\right)^{1/p},
\end{align*} 
 where $L_t$ is the side length of $U_t$ and $C_{n,p}$ is the constant in (\ref{C on cubes}). Now, observe that  $L_t\leq d(x)$ for all $x\in U_t$. Indeed, if $x\in Q_t$ then 
 \[L_t=\frac{_9}{^8}l_t<\sqrt{n}\,l_t=\diam(Q_t)\leq \dist(Q_t,\partial\Omega)\leq d(x),\]
 where $l_t$ is the side length of $Q_t$. Now, if $x \in U_t\setminus Q_t$ then 
\[\sqrt{n}\,l_t\leq \dist(Q_t,\partial\Omega)\leq \dist(U_t,\partial\Omega) + \frac{_1}{^{16}} \sqrt{n}\, l_t,\]
hence, $\frac{15}{16}\sqrt{n}\,l_t\leq \dist(U_t,\partial\Omega)$ and 
\[L_t=\frac{_9}{^8}l_t<\frac{_{15}}{^{16}}\sqrt{n}\,l_t\leq\dist(U_t,\partial\Omega)\leq d(x).\]

Then, the validity of $L_t\leq d(x)$ for all $x\in U_t$ implies (\ref{local Frac}) for all $U_t$, where $\mu(x,y)$ is the 
function
defined in (\ref{mu}), and the uniform constant 
 \begin{equation}\label{C1}
 C_1=C_{n,p} \tau^{s-n} 2^\beta.
 \end{equation}
 
Next, by Theorem \ref{Decomp}, there is a  finite  $\C$-decomposition of functions $\{g_t\}_{t\in\Gamma}$  subordinate to $\{U_t\}_{t\in\Gamma}$ of any function $g$ in $\W_0$  which satisfies (\ref{P02}) and $(\ref{P01})$. Moreover, using (\ref{Eccentricity}), it can be seen that  
\begin{equation*}
|g_t(x)|\leq |g(x)|+(72K)^nTg(x),
\end{equation*}
for all $t\in\Gamma$ and $x\in U_t$. 

Now, $\omega(x):=d_F^\beta(x)$ fulfills the hypothesis of Lemma \ref{weighted T} where the constant in (\ref{weight}) is $C_2=(3K \sqrt{n})^\beta$. The last assertion uses Lemma \ref{dist F}. Thus, the operator $T$ is continuous from $L:=L^q(\Omega,d_F^{-q\beta})$ to itself with   the norm 
\[\|T\|_{L\to L} \leq 2\left(\frac{qN}{q-1}\right)^{1/q}(3K\sqrt{n})^\beta.\]
Hence, 
\begin{align*}
&\sum_{t\in \Gamma} \|g_t\|^q_{L^q(U_t,d_F^{-q\beta})}\\
&\leq 2^{q-1} \left\{\left(\sum_{t\in\Gamma} \int_{U_t}|g(x)|^q\, d_F^{-q\beta}(x)\right)+(72K)^{qn}\left(\sum_{t\in\Gamma} \int_{U_t}|Tg(x)|^q\, d_F^{-q\beta}(x)\right)\right\}\\
&\leq 2^{q-1}N \left\{\int_\Omega |g(x)|^q\, d_F^{-q\beta}(x)\,\dx+(72K)^{qn}\int_\Omega |Tg(x)|^q\, d_F^{-q\beta}(x)\,\dx\right\}\\
&\leq 2^{q-1}N \left\{1+(72K)^{qn}2^q\left(\frac{qN}{q-1}\right)(3K \sqrt{n})^{q\beta}\right\} \|g\|^q_{L^q(\Omega,d_F^{-q\beta})}\\
&\leq 4^q N^2 (72K)^{qn}\left(\frac{q}{q-1}\right)(3K \sqrt{n})^{q\beta} \|g\|^q_{L^q(\Omega,d_F^{-q\beta})}\\
&= 4^q\,12^{2n}\,72^{qn}\,(3 \sqrt{n})^{q\beta}\left(\frac{q}{q-1}\right) K^{q(n+\beta)} \|g\|^q_{L^q(\Omega,d_F^{-q\beta})}.
\end{align*}
Therefore, we have a $\C$-decomposition subordinate to $\{U_t\}_{t\in\Gamma}$ with constant   
 \begin{equation}\label{C0}
C_0=4 (12)^{2n/q} (72)^n (3 \sqrt{n})^{\beta}\left(\frac{q}{q-1}\right)^{1/q} K^{n+\beta}.
\end{equation}
Finally, inequality (\ref{lemma Frac}) and the control  on the overlapping of the tree covering by $N=12^n$ implies (\ref{Fractional}). 
\end{proof}

 \begin{remark}\label{constant}
Notice that the proof of Theorem \ref{Main} provides an explicit constant $C=2C_0C_1$ for inequality \eqref{Fractional}, where $C_0$ and $C_1$ are described respectively in \eqref{C0} and \eqref{C1}. 
\end{remark}

The next result, similar to Proposition \ref{DD}, follows from the H\"older inequality (equivalently, from Minkowski's integral inequality).

\begin{proposition}\label{DD rho} Let $\rho:\R^n\setminus\{{\bf 0}\}\to\R$ be a positive radial Lebesgue measurable function which is increasing with respect to the radius. Then, 
the fractional Poincar\'e type inequality 
\begin{align}\label{C on DD rho}
\|u(x)-u_U\|_{L^p(U)}\leq \frac{{\rm diam}\,(U)^{n/p} \, \rho({\rm diam}(U))}{|U|^{1/p}}\left( \int_U\int_U \frac{|u(y)-u(x)|^p}{|y-x|^n(\rho|y-x|)^p}\,\d y\d x\right)^{1/p}
\end{align}
holds for any bounded domain $U$ in $\R^n$ and $1 < p<\infty$, where $u_U:=\frac{1}{|U|}\int_U u(y)\dy$.
\end{proposition}

\begin{proof} 
\begin{align*}
\int_{U}|u(x)-u_U|^p  \dx &= \int_{U}\left|\frac{1}{|U|}\int_U u(x)-u(y) \dy\right|^p \dx
\leq \frac{1}{|U|}\int_U\int_U |u(x)-u(y)|^p \d y\d x\\
&\leq \dfrac{\diam(U)^n\{\rho(\diam(U))\}^p}{|U|}\int_U\int_U \dfrac{|u(x)-u(y)|^p}{|x-y|^n(\rho|x-y|)^p} \d y\d x.\\
\end{align*}
\end{proof}

\begin{remark} If $\rho(x)=|x|^s$, with $s\in (0,1)$, we recover the classical fractional Poincar\'e inequality.  
\end{remark}

We generalize the fractional Poincar\'e inequality stated in Theorem \ref{Main} by replacing the fractional derivatives given by the power functions  $|x|^s$, with $0<s<1$, by general increasing and positive radial functions $\rho|x|$.

\begin{theorem}\label{Main 2} Let $\O$ in $\R^n$ be a bounded John domain and $1< p<\infty$. Given an arbitrary compact set $F$ in $\partial\Omega$, a parameter $\beta\geq 0$ and a positive radial Lebesgue measurable function $\rho:\R^n\setminus\{{\bf 0}\}\to\R$ increasing with respect to the radius, there exists a constant $C$ such that 
\begin{multline}\label{Fractional 2}
\left(\int_{\Omega}|u(x)-{ u_{\Omega ,\omega }}|^p d_F^{p\beta}(x)\dx\right)^{1/p}\\
\leq C \left(\int_\Omega \int_{\O\cap B(x,d(x))}\frac{|u(x)-u(y)|^p}{|x-y|^n(\rho|x-y|)^p}[\rho(2d(x))]^p d_F^{p\beta}(x)\,\dy\dx\right)^{1/p}
\end{multline}
for all function $u\in L^p(\O,d(x)^{p\beta})$. We denote by $d(x)$ and $d_F(x)$ the distance from $x$ to $\partial\O$ and $F$ respectively, and by $u_{\Omega ,\omega}$ the weighted average $\frac{1}{d_F^{p\beta}(\O)}\int_\O u(z) d_F^{p\beta}(z)\dz$.

In addition, the constant $C$ in (\ref{Fractional 2}) can be written as 
\begin{align*}
C= C_{n, p, \beta}\, K^{n+\beta},
\end{align*} 
where $K$ is the geometric constant introduced in (\ref{Boman tree}).
\end{theorem}

\begin{proof} This proof mimics the one of Theorem \ref{Main} with Proposition \ref{DD rho} instead of Proposition \ref{on cubes}. Indeed, we will use again the tree covering $\{U_t\}_{t\in\Gamma}$ of $\Omega$ defined in (\ref{cover John}) and the weight $\omega(x)=d_F^\beta(x)$, however, in this case $\mu(x,y)$ is defined as 
\begin{align*}
\mu(x,y):=\frac{[\rho(2d(x))]^p d_F^{p\beta}(x)}{|x-y|^n(\rho|x-y|)^p}.
\end{align*}
We only have to show that (\ref{local Frac}) is verified on $U_t$, for all $t$, with uniform constant. This fact follows from (\ref{C on DD rho}) by using  the inequality $\diam(U_t)\leq 2 d(x)$ for all $x\in U_t$.
\end{proof} 

 \section*{Acknowledgements}
This research was initiated when the second author visited the University of Helsinki during the Summer of 2016. This author gratefully acknowledges Professor Michel Lapidus from University of California Riverside for his financial support for the expenses of the trip.

\end{document}